\documentclass[12pt,oneside]{amsart}

\usepackage{amssymb,amsmath}
\usepackage{MnSymbol}

\numberwithin{equation}{section}
\usepackage{tikz}
\usetikzlibrary{arrows,shapes,positioning}
\tikzstyle{every loop}= []

\colorlet{myGray}{gray!25}

\tikzset{my circle/.style={circle,draw=black,fill=myGray,inner
    sep=0pt,minimum size=6pt}} \tikzset{my square/.style={regular
    polygon,regular polygon sides=4,draw=black,fill=myGray,inner
    sep=0pt,minimum size=9pt}} \tikzset{my star/.style={star,star
    point ratio=2.5,draw=black,fill=myGray,inner sep=0pt,minimum
    size=9pt}} \tikzset{my triangle/.style={regular polygon,regular
    polygon sides=3,draw=black,fill=myGray,inner sep=0pt,minimum
    size=9pt}} \tikzset{my
  kite/.style={diamond,aspect=0.25,draw=black,fill=myGray,inner
    sep=1pt,minimum size=6pt}}
\tikzset{every pin/.style={pin distance=3pt,inner sep=1pt,font=\tiny}}
\tikzset{every pin edge/.style={semithick}}
\newtheorem{theorem}{Theorem}[section]
\newtheorem{lemma}[theorem]{Lemma}

\theoremstyle{definition}
\newtheorem{definition}[theorem]{Definition}

\newtheorem{remark}[theorem]{Remark}
\newtheorem{example}[theorem]{Example}

\numberwithin{equation}{section}

\newcommand{\PP}{\mathbb{P}}

\begin{document}

%%%%%%%%%%%%%%%%%%%%%%%%%%%%%%%%%%%%%%%%%%%%%%%%%%%%%%%%%%%%%%%%%%%%%%%%

\title{Hadamard Star Configurations}
\thanks{Last updated: Jan. 14, 2018}

\author[E. Carlini]{Enrico Carlini}
\address[E. Carlini]{DISMA-Department of Mathematical Sciences \\
Politecnico di Torino, Turin, Italy}
\email{enrico.carlini@polito.it}

\author[M. V. Catalisano]{Maria Virginia Catalisano}
\address[M. V. Catalisano]{Dipartimento di Ingegneria Meccanica, Energetica, Gestionale e dei
Trasporti, Universit\`a degli studi di Genova, Genoa, Italy}
\email{catalisano@dime.unige.it}

\author[E. Guardo]{Elena Guardo}
\address[E. Guardo]{Dipartimento di Matematica e Informatica\\
Universit\`a degli studi di Catania\\
Viale A. Doria, 6 \\
95100 - Catania, Italy}
\email{guardo@dmi.unict.it}

\author[A. Van Tuyl]{Adam Van Tuyl}
\address[A. Van Tuyl]{Department of Mathematics and Statistics\\
McMaster University, Hamilton, ON, L8S 4L8}
\email{vantuyl@math.mcmaster.ca}


\begin{abstract}
Bocci, Carlini, and Kileel have shown that the square-free
Hadamard product of a finite set of points $Z$ that
all lie on a line $\ell$ in $\mathbb{P}^n$ produces
a star configuration of codimension $n$.
In this paper
we introduce a construction using the Hadamard product to construct
star configurations of codimension $c$.  In the
case that $c = n= 2$, our construction produces the
star configurations of Bocci, Carlini, and Kileel.
We will call any star configuration
that can be constructed using our approach a Hadamard star
configuration.
Our main result is a classification of
Hadamard star configurations.
\end{abstract}

\maketitle

%%%%%%%%%%%%%%%%%%%%%%%%%%%%%%%%%%%%%%%%%%%%%%%%%%%%%%%%%%%%%%%%%%%%%%%%%%%%%%%

\section{Introduction}\label{sec:intro}

During the past decade, star configurations in $\mathbb{P}^n$ have been
identified as an interesting family of projective varieties.
Roughly speaking (formal definitions are postponed until the next section),
a star configuration is a union of linear
spaces that have been intersected in a prescribed manner.   Under
the appropriate conditions, a star configuration can be a set of
points in $\mathbb{P}^n$.   These
varieties and their corresponding ideals have been found to have
interesting extremal properties (e.g., in relation to the
containment problem for symbolic powers of ideals \cite{BH}).  At the same
time, there are nice descriptions of the minimal generators and
minimal graded free resolutions of these ideals (see \cite{GHM}).
Other papers that have studied star configurations include
\cite{CGVT,GHMN,PS:1}.

Quite recently (and perhaps, surprisingly) it has been
shown that star configurations of points also arise when one studies the
Hadamard product of projective varieties.  Given two varieties
$X$ and $Y$ in $\mathbb{P}^n$, the Hadamard product of $X$ and
$Y$, denoted $X \star Y$, is the closure of the rational map
from $X \times Y$ to $\mathbb{P}^n$ given by
\[([a_0:\cdots:a_n], [b_0:\cdots:b_n]) \rightarrow [a_0b_0:\cdots:a_nb_n].\]
Note that the name is inspired by the Hadamard product of two matrices
in linear algebra where one multiplies two matrices of the same size
entry-wise.   The notion of a Hadamard product of varieties first
appeared in \cite{CMS,CTY} in the study of the geometry
of Boltzmann machines.  The Hadamard product is also
a useful tool in tropical geometry (see \cite[Theorem 5.5.11]{MS}).
 More recently Bocci, Kileel, and the first
author \cite{BCK} developed some of the properties of Hadamard
products of linear spaces.  Additional results on Hadamard
products can be found in \cite{BCFL,BCFL2,FOW}.
 It was in the paper
of Bocci, Carlini, and Kileel (see \cite[Theorem 4.7]{BCK})
that a connection to star configurations was made.  In particular,
if $\ell$ is a line in $\mathbb{P}^n$ and $Z$ is a set of
points on $\ell$, then under some suitable hypotheses, the Hadamard
product of $Z$ with itself at most $ \min\{n,|Z|\}$ times is
a star configuration of points.

It is natural to ask if the Hadamard product can be used
to construct star configurations of various codimensions, not
just codimension $n$.   We introduce such a construction
in this paper.
A star configuration that can be constructed via our approach will be
called a {\it Hadamard star configuration}.  When $n=2$,
we show that our Hadamard star configurations are the
star configurations produced by Bocci, et al.'s procedure.

It can be shown that there are star configurations that cannot be
Hadamard star configurations (see Example \ref{nonexample}).  It thus
behooves us to ask if one can determine if a given star configuration
is a Hadamard star configuration.  One of the main results
of this paper is a classification of Hadamard star
configurations (see Theorem \ref {thmhadamardset} and
Remark \ref {remhadamardset}).

Our paper is structured as follows.
In the next section
we present the definitions and results mentioned in the introduction.
In particular, we introduce Hadamard star configurations.
In section three we present  our
classification result.
In the final section, we examine the problem of constructing
Hadamard star configurations.

\noindent
{\bf Acknowledgments.}
The first and second authors
{ thank McMaster University for their support while
visiting the fourth  author}.
The  first and second authors were also supported by GNSAGA of INDAM funds.
The fourth author acknowledges the financial support provided by NSERC.

%%%%%%%%%%%%%%%%%%%%%%%%%%%%%%%%%%%%%%%%%%%%%%%%%%%%%%%%%%%%%%%%%%%%%%%%%

\section{Background results}\label{sec:background}

In this section, we introduce the relevant background.
Throughout this paper $S = \mathbb{C}[x_0,\ldots,x_n]$.  We will
denote by $S_i$ the homogeneous degree $i$ piece of $S$. Given $I\subset S$,
a homogeneous ideal, we let
$V(I)\subset\PP^n$ denote the variety defined by the vanishing locus
of all elements of $I$.

We begin by defining a star configuration.

\begin{definition}
A set $\mathcal{L}=\{L_1,\ldots,L_r\}\subset S_1$ is {\it generally linear}
if $r\geq n+1$ and
{ if any  choice of  $n+1$ distinct elements in $\mathcal{L}$
are  linearly independent.}

\end{definition}

\begin{definition}
Let $\mathcal{L}=\{L_1,\ldots,L_r\}\subset S_1$ and let $c$
be an integer $1\leq c\leq n$. We introduce the variety
\[\mathbb{X}_c(\mathcal{L})=\bigcup_{1\leq i_1<\cdots<i_c\leq r}
V(L_{i_1},\ldots,L_{i_c})\subset\PP^n.\]
If $\mathcal{L}$ is generally linear we say that
$\mathbb{X}_c(\mathcal{L})$ is a {\it codimension $c$ star configuration},
or simply, a {\it star configuration}.
\end{definition}

\begin{remark}
Note that the star configuration $\mathbb{X}_c(\mathcal{L})$
is completely determined by $c$ and by the points
$[L_1],\ldots,[L_r]\in\PP S_1$.   Alternatively,
given $\mathcal{L}=\{L_1,\ldots,L_r\}$, we can consider the hyperplanes
$H_i=V(L_i)$.  These
hyperplanes uniquely determine $\mathbb{X}_c(\mathcal{L})$,
thus we will sometimes write $\mathbb{X}_c(H_1,\ldots,H_r)$.
\end{remark}

\begin{example}
Let $\mathcal{L} = \{L_1,\ldots,L_5\} \subset S_1$
where $S = \mathbb{C}[x_0,x_1,x_2]$ be a generally linear set.
If $c = 2$, then the star configuration $\mathbb{X}_2(\mathcal{L}) \subset{\Bbb P}^2$
is the 10 points of intersection of the five lines, as shown in
Figure \ref{starconfigpicture}.   If $c = 1$, then
the star configuration $\mathbb{X}_1(\mathcal{L})$ is the
union of the five lines in Figure \ref{starconfigpicture}.  Notice that
the lines resemble a star in this case.
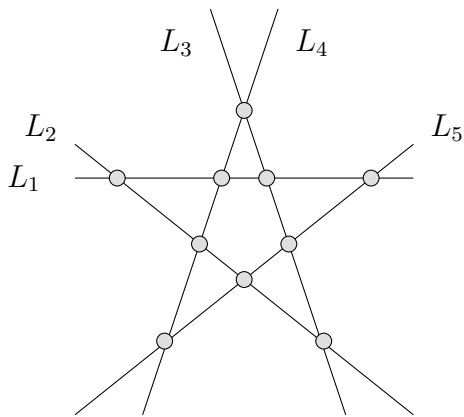
\begin{figure}[!h]
  \centering
    \begin{tikzpicture}[scale=0.45]
      \draw (-5,3)--(5,3);
      \draw (1,8)--(-3,-4);
      \draw (-1,8)--(3,-4);
      \draw (5,4)--(-5,-4);
      \draw (-5,4)--(5,-4);
      \node at (-6.5,3) {$L_1$};
      \node at (-6,4.5) {$L_2$};
      \node at (-2,7) {$L_3$};
      \node at (2,7)  {$L_4$};
      \node at (6,4.5) {$L_5$};
      \node[my circle] at (0,5) {};
      \node[my circle] at (-.666,3) {};
      \node[my circle] at (.666,3) {};
      \node[my circle] at (3.75,3) {};
      \node[my circle] at (-3.75,3) {};
      \node[my circle] at (0,0) {};
      \node[my circle] at (-2.35,-1.82) {};
       \node[my circle] at (2.35,-1.82) {};
      \node[my circle] at (-1.32,1.05) {};
      \node[my circle] at (1.32,1.05) {};
    \end{tikzpicture}
    \caption{A star configuration of 10 points in $\mathbb{P}^2$}
    \label{starconfigpicture}
\end{figure}
\end{example}

We now turn our attention to the required background on Hadamard products.

\begin{definition}\label{defn:hadamardpoints}
Let $P = [a_0:\cdots:a_n]$ and $Q = [b_0:\cdots:b_n]$ be two
points in $\mathbb{P}^n$.   If there exists some $i$ such that
$a_ib_i \neq 0$, then the {\it Hadamard product} of $P$ and $Q$,
denoted $P \star Q$, is given by
\[P \star Q = [a_0b_0: \cdots : a_nb_n].\]
Otherwise, we say $P \star Q$ is not defined.
\end{definition}

\begin{definition}
The {\it Hadamard product of two varieties} $X,Y \subset \mathbb{P}^n$ is
given by
\[X \star Y = \overline{\{ P \star Q ~|~ P \in X, Q \in Y, ~~\mbox{and
$P \star Q$ is defined}\}},\]
where we mean the closure in $\mathbb{P}^n$ with respect to the
Zariski topology.  The $r$-th {\it Hadamard power} of $X$ is $X^{\star r}
= X \star X^{\star (r-1)}$ where $X^0 = [1:\cdots:1]$.
\end{definition}

\begin{remark}  The Hadamard product can be defined in terms  of
the Segre product and a specific projection map, as in
\cite[Definition 2.2]{BCK}.  One
can show that our definition is equivalent to the definition
given above (see \cite[Remarks 2.4 and 2.5]{BCK}).  The defining
ideal of $I(X \star Y)$ can be computed using elimination;
see \cite[Remark 2.6]{BCK} for an algorithm.
\end{remark}

When using the Hadamard product, we sometimes need to ensure
we have enough non-zero coordinates in our points.  The following
notation shall prove useful.

\begin{definition}
Let  $\Delta_i$ be the set of points of $\mathbb{P}^n$
which have at most $i+1$ non-zero coordinates, that is,  $\Delta_i$ is the set of points
which have at least  $n-i$ zero coordinates. In particular, $\Delta_n =
\mathbb{P}^n$.

%For $j=0,\ldots,n$, let $H_i$ denote the hyperplane
%$in $\mathbb{P}^n$ given by $H_j = V(x_j)$.
%We set
%\[\Delta_i = \bigcup_{0 \leq j_1 < \cdots < j_{n-i} \leq n} H_{j_1} \cap \cdots \cap
%H_{j_{n-i}}.\]
\end{definition}

We now introduce a special set of linear forms.

\begin{definition}
A set of linear forms $\mathcal{L}=\{L_1,\ldots,L_r\}\subset S_1$ is
a {\it Hadamard set} if there exist  an $L=a_0x_0+\cdots+a_nx_n\in S_1$ and
points $P_1,\ldots,P_r\in\PP^n$ with $P_i=[p_0(i):\cdots:p_n(i)]
\in \mathbb{P}^n \setminus \Delta_{n-1}$,   that is, $p_j(i) \neq 0$ for any $i,j$, such that
\[[L_i]=
\left[{a_0 \over p_0(i)}x_0+\cdots+{a_n\over p_n(i)}x_n\right]\in\PP S_1.\]
We say that $\mathcal{L}$ is a {\it strong Hadamard set} if, in addition,
\[\sum_{j=0}^{n} a_jp_j(i)=0~~\mbox{for $i=1,\ldots,r$}.\]
\end{definition}

\begin{remark} \label{support}  Recall that the support of a linear form $L$ is the set of variables appearing in $L$ with a non-zero coefficient. The set  $\mathcal{L}=\{L_1,\ldots,L_r\}$ is a
Hadamard set if and only if all the linear forms $L_i$ have
the same support. To see why, just define $L$ with the same support as the elements of $\mathcal{L}$
and all coefficients equal one;
the $P_i$ are then defined by taking the inverse of the
coefficients of the linear forms $L_i$. However a necessary and sufficient condition to be a strong Hadamard set is more complicated, as shown in Theorem \ref{thmhadamardset}.
\end{remark}

We are now able to define our main object of interest.

\begin{definition}
A star configuration $\mathbb{X}_c(\mathcal{L})$ is a {\it Hadamard star configuration} if $\mathcal{L}$ is a strong Hadamard set.
\end{definition}

\begin{example}\label{nonexample}
Not every star configuration is a Hadamard star configuration.  For example,
let $S = \mathbb{C}[x_0,x_1,x_2]$ and consider $\mathcal{L} = \{x_0,x_1,x_2\}$.
It is easy to see that the star configuration
\[\mathbb{X}_2(\mathcal{L}) = \{[1:0:0],[0:1:0],[0:0:1]\},\]
which consists of the three coordinate
points of $\mathbb{P}^2$, is not a Hadamard star configuration.
Indeed, by Remark \ref{support}, $\mathcal{L}$ is not even a Hadamard  set.
\end{example}

We use the name Hadamard because of the connection with
the Hadamard product, as shown in the next two lemmas.

\begin{lemma}\label{hadamardpoint}
Let $L = a_0x_1 + \cdots + a_nx_n \in S_1$ and $P = [p_0:\cdots:p_n]
\in \mathbb{P}^n \setminus \Delta_{n-1}$.  Then
\[P \star V(L) = V(L') ~~\mbox{where $L' = \frac{a_0}{p_0}x_0 + \cdots
+ \frac{a_n}{p_n}x_n$}.\]
\end{lemma}

\begin{proof}
For any $Q = [q_0:\cdots:q_n]
\in V(L)$, the linear form
$L'$  vanishes on the point $P \star Q = [p_0q_0:\cdots:p_nq_n]$;  in fact
\[ \frac{a_0}{p_0}(p_0q_0) + \cdots
+ \frac{a_n}{p_n}(p_nq_n) = a_0q_0+\cdots+a_nq_n = 0.\]

Conversely, for any $R = [r_0:\cdots:r_n] \in V(L')$, the equation
\[\frac{a_0}{p_0}(r_0) + \cdots
+ \frac{a_n}{p_n}(r_n) = 0\]
implies $R' = [\frac{r_0}{p_0}:\cdots:\frac{r_n}{p_n}] \in V(L)$.
Since $R = P \star R'$, we get
$R  \in P \star V(L)$.
\end{proof}

\begin{lemma} \label{charlemma}
Let $\mathcal{L} = \{L_1,\ldots,L_r\} \subset S_1$.
\begin{enumerate}
\item[$(i)$] $\mathcal{L}$ is a Hadamard set for $L$ and
$P_1,\ldots,P_r$ if and only if $V(L_i) = P_i \star V(L)$ for all $1\leq i \leq r$.
\item[$(ii)$] $\mathcal{L}$ is a strong Hadamard set for $L$ and
$P_1,\ldots,P_r$ if and only if $V(L_i) = P_i \star V(L)$
and $P_i \in V(L)$ for all $1\leq i \leq r$.
\end{enumerate}
\end{lemma}

\begin{proof} We prove only $(ii)$ since $(i)$ can also be obtained from
this proof.  Let $\mathcal{L} = \{L_1,\ldots,L_r\}$ be a strong Hadamard set
for $L \in S_1$ and $P_1,\ldots,P_r \in \mathbb{P}^n \setminus \Delta_{n-1}$.
If $L = a_0x_0+\cdots+a_nx_n$ and $P_i = [p_0(i):\cdots:p_n(i)]$,
by definition of  a Hadamard set,
 $L_i = \frac{a_0}{p_0(i)}x_0+\cdots+\frac{a_n}{p_n(i)}x_n$ for each $i$.
Lemma \ref{hadamardpoint} then implies $V(L_i) = P_i \star V(L)$.  The
condition $\sum_{j=0}^{n} a_jp_j(i)=0$ is simply the condition that
each $P_i \in V(L)$.  Reversing this argument gives the reverse implication.
\end{proof}

%\begin{remark}
%Note that, $\mathcal{L}=\{L_1,\ldots,L_r\}$ is an Hadamard set for $L$ and $P_1,\ldots,P_r$ if and only if the hyperplanes $H_i=V(L_i)$ and $H=V(L)$ are such that $H_i=P_i\star H$ for all $i$. In particular, $\mathcal{L}$ is a strong Hadamard set if, in addition, $P_1,\ldots,P_r\in H$.
%\end{remark}

We round out this section by explaining the connection to
\cite{BCK}.   We first need the following variant of the Hadamard product.

\begin{definition}\label{sqfreehadamard}
If $\mathbb{X}$ is a finite set of points in $\mathbb{P}^n$, then
the $r$-th {\it square-free Hadamard product} of $\mathbb{X}$
is
\[\mathbb{X}^{\underline{\star} r} = \{P_1 \star \cdots \star P_r ~|~
P_i \in \mathbb{X} ~~\mbox{and}~~ P_i \neq P_j \}.\]
\end{definition}

Bocci, Carlini, and Kileel then proved the following result
(we have specialized their result).

\begin{theorem}[{\cite[Theorem 4.7]{BCK}}]\label{starconfighadamard}
Let $\ell$ be a line in $\mathbb{P}^n$ such that $\ell \cap \Delta_{n-2}
= \emptyset$, and let $\mathbb{X} \subset \ell$ be a set of $m > n$
points with $\mathbb{X} \cap \Delta_{n-1} = \emptyset$.
Then $\mathbb{X}^{\underline{\star} n}$ is
a star configuration of $\binom{m}{n}$ points of $\mathbb{P}^n$.
\end{theorem}

When $n=2$, we now show that the star configurations produced
by the above result are always Hadamard star configurations.
Moreover, even if $n \neq 2$, if we assume one extra condition on the
line $\ell$, we can create more Hadamard star configurations. We have the following:

\begin{theorem}
Let $\ell$ be a line in $\mathbb{P}^n$ such that $\ell \cap \Delta_{n-2}
= \emptyset$, and let $\mathbb{X} \subset \ell$ be a set of $m > n$
points with $\mathbb{X} \cap \Delta_{n-1} = \emptyset$.
\begin{enumerate}
\item[$(i)$] If $[1:1:\cdots:1] \in \ell$, then
$\mathbb{X}^{\underline{\star} n}$ is a Hadamard star
configuration.
\item[$(ii)$] If $n=2$, then  $\mathbb{X}^{\underline{\star} 2}$
is a Hadamard star configuration.
\end{enumerate}
\end{theorem}

\begin{proof}   Let $\mathbb{X} = \{P_1,\ldots,P_m\} \subset \ell$ with
$m > n$.

$(i)$.  If $[1:\cdots:1] \in \ell$, then
$\ell \subset \ell^{\star 2} \subset \ell^{\star 3} \subset \cdots \subset
\ell^{\star (n-1)}$  (see, for example, \cite{FOW} after Definition 4.6).
By \cite[Corollary 3.7]{BCK}, the variety $\ell^{\star (n-1)}$ is
a hyperplane defined by one linear form, say
\[\ell^{\star (n-1)} = V(L) ~~\mbox{where $L = a_0x_0 + \cdots + a_nx_n$.}\]
Note that since $\mathbb{X} \subset \ell \subset \ell^{\star (n-1)}$,
every point $P \in \mathbb{X}$ satisfies $L(P) = 0$.

For $i=1,\ldots,m$, let $L_i$ be the linear form such that
$V(L_i) = P_i \star V(L)$.  As shown in the proof
of \cite[Theorem 4.7]{BCK}, these linear forms are  generally
linear.    Moreover, $\mathbb{X}^{\underline{\star} n}$ is the star
configuration given by $\mathcal{L} = \{L_1,\ldots,L_m\}$.
We have now shown that $\mathcal{L}$ is a strong Hadamard set by
Lemma \ref{charlemma}.

$(ii)$.  When $n=2$, then $\ell= \ell^{\star (n-1)}$ is given by a linear form
$L=a_0x_0 + \cdots + a_nx_n$.  We can then repeat the above argument.
\end{proof}

\begin{remark}
If $[1:\cdots:1] \not\in \ell \subset \mathbb{P}^n$, then
$\ell^{\star (n-1)}$ is still given by a linear form $L$, and
$\mathbb{X}^{\underline{\star} n}$ is still the star configuration given
by $\mathcal{L} = \{L_1,\ldots,L_m\}$ where $V(L_i)
= P_i \star V(L)$.  So, in particular, $\mathcal{L}$ is a Hadamard
set.  However, it may not be a strong Hadamard set because
there is no guarantee that $L$ vanishes on every $P_i \in \mathbb{X}$.
For instance, let $\ell \subset \Bbb P^3$ be the line through $[0:1:2:1]$, $[0:1:2:1]$ and $[2:-1:0:1]$.
We have that  $\ell^{\star 2}$ is the plane through
$[0:0:2:1] $,  $[0:-1:0:1] $ and $[2:0:0:1] $, and this plane does not contain the line $\ell$.
\end{remark}

%%%%%%%%%%%%%%%%%%%%%%%%%%%%%%%%%%%%%%%%%%%%%%%%%%%%%%%%%%%%%%%%%%%%%%%%%

\section{Characterization of Hadamard star configurations}
\label{sec:charcterization}

As shown in Example \ref{nonexample}, not every star configuration
can be a Hadamard star configuration.  It is then natural to ask
if we can classify what star configurations are Hadamard star configurations.
The main result of this section is a solution to this question.

The classification of Hadamard star configurations reduces to classifying
strong Hadamard sets.

\begin{theorem}\label{thmhadamardset}
Let $\mathcal{L}=\{L_1,\ldots,L_r\}$ be a generally linear set where
$L_i=a_0(i)x_0+\cdots+a_n(i)x_n$ for $i=1,\ldots,r$.
Then $\mathcal{L}$ is a strong Hadamard set if and only if
\begin{enumerate}
\item\label{thmhadamardset1} $a_j(i)\neq 0$ for all $i,j$; and
\item\label{thmhadamardset2}  there exists an $\mathbf{a}\in\mathbb{C}^{n+1}$ without zero coordinates such that $M\mathbf{a}= {\bf 0}$, where
\[M=\left(\begin{array}{ccc}{1\over a_0(1)} & \ldots & {1\over a_n(1)} \\ \vdots & \vdots & \vdots\\ {1\over a_0(r)} & \ldots & {1\over a_n(r)}\end{array}\right). \]

\end{enumerate}

 Moreover, if $n=2$,  condition  $(2)$ is equivalent to
the condition that  $\mathrm{rk} (M)\leq 2$.
\end{theorem}
\begin{proof}
Assume that $\mathcal{L}$ is a strong Hadamard set. Thus  there exists a linear
form $L = a_0x_0+\cdots +a_nx_n$ and $r$ points $P_i = [p_0(i):\cdots:
p_n(i)] \in \mathbb{P}^n \setminus \Delta_{n-1}$ such that
$P_i \star V(L) = V(L_i)$ for $i=1,\ldots,r$.  By Lemma \ref{hadamardpoint},
this implies that
\[a_j(i) = \frac{a_j}{p_j(i)} ~~\mbox{for all $1 \leq i \leq r$ and
$0 \leq j \leq n$.}\]

Since $\mathcal  L$ is generally linear, then $a_j \neq 0$ for all $j=0,\ldots,n$.  To see this, assume for a contradiction,  that $a_n = 0$.  Then
the above expression for $a_n(i)$ implies $a_n(i) = 0$ for all $i=1,\ldots,r$.
Thus $L_1,\ldots,L_r$ are linear forms not involving
 $x_n$, and hence any $n+1$ elements
of $\{L_1,\ldots,L_r\}$ are linear dependent, contradicting our
assumption that the elements of $\mathcal{L}$ are generally linear.

Because $a_j \neq 0$ for all $j=0,\ldots,n$, we thus have that
$a_j(i) \neq 0$ for all $1 \leq i \leq r$ and
$0 \leq j \leq n$, thus proving $\eqref{thmhadamardset1}$.

For each $i=1,\ldots,r$, we also have $P_i \in V(L)$ since $\mathcal{L}$
is a strong Hadamard set.  Since $p_j(i) = \frac{a_j}{a_j(i)}$ for
all $1 \leq i \leq r$ and
$0 \leq j \leq n$, we have
\[a_0p_0(i)+\cdots +a_np_n(i) = \frac{a_0^2}{a_0(i)}+\cdots + \frac{a_n^2}{a_n(i)}
 = 0\]
for all $i=1,\ldots,r$.  But this means that
${\bf a} = \begin{pmatrix} a_0^2 & a_1^2 & \cdots & a_n^2 \end{pmatrix}^T \in
\mathbb{C}^{n+1}$ satisfies $M{\bf a} = {\bf 0}$.  Since each $a_j^2 \neq 0$ for $j=0,\ldots,n$,
the vector ${\bf a}$ shows that $\eqref{thmhadamardset2}$ is
satisfied.

Now, by assuming that $\eqref{thmhadamardset1}$ and $\eqref{thmhadamardset2}$ hold, we need to find a linear form $L = a_0x_0+\cdots + a_nx_n$
and points $P_i = [p_0(i):\cdots:p_n(i)] \in \mathbb{P}^n \setminus
\Delta_{n-1}$ for $i=1,\ldots,r$,
such that $V(L_i) = P_i \star V(L)$ for $i=1,\ldots,r$ and
$P_i \in V(L)$ for $i=1,\ldots,r$.

{
By $\eqref{thmhadamardset2}$, there exists a vector
 ${\bf a} \in \mathbb{C}^{n+1}$ with no zero coordinates
such that $M{\bf a}  = {\bf 0}$.
Say  ${\bf a}^T = \begin{pmatrix} b_0 & \cdots & b_n \end{pmatrix}$,
and let $a^2_i = b_i$;  note that because $\mathbb{C}$ is algebraically
closed, each
$a_i \in \mathbb{C}$.}

 We now claim
that $L = a_0x_0+\cdots + a_nx_n$ and
\[P_i = \left[\frac{a_0}{a_0(i)}:\frac{a_1}{a_1(i)}:\cdots:\frac{a_n}{a_n(i)}\right]
~~\mbox{for $i=1,\ldots,r$}\]
are the desired linear form and points.  We first note that $\eqref{thmhadamardset1}$ implies
that the points $P_i$ are defined.  As well, each $a_j \neq 0$,
so each $P_i \in \mathbb{P}^n \setminus \Delta_{n-1}$.
By Lemma \ref{hadamardpoint}, $P_i \star V(L) = V(L_i)$
for each $i$ since
\[L_i = a_0(i)x_0+\cdots+a_n(i)x_n = \frac{a_0}{\frac{a_0}{a_0(i)}}x_0
+ \cdots + \frac{a_n}{\frac{a_n}{a_n(i)}}x_n.\]
Finally, each $P_i \in V(L)$ since
\[a_0\left(\frac{a_0}{a_0(i)}\right) + \cdots + a_n\left(\frac{a_n}{a_n(i)}\right)
= \frac{a^2_0}{a_0(i)}+ \cdots + \frac{a^2_n}{a_n(i)} = 0,\]
where the final equality follows from the fact that
${\bf a}^T =
\begin{pmatrix} a_0^2 & \cdots & a_n^2 \end{pmatrix}$
is a solution to $M{\bf a} = {\bf 0}$.

To prove the final statement assume   $n=2$.
Condition $(1)$ implies that  the matrix $M$ given
below exists:
\[M=\left(\begin{array}{ccc}{1\over a_0(1)} & {1\over a_1(1)} & {1\over a_2(1)} \\ \vdots & \vdots & \vdots\\ {1\over a_0(r)} & {1 \over a_1(r)} &
{1\over a_2(r)}\end{array}\right).\]
Obviously the condition $(2)$ implies that $rk(M) \leq 2$.
Now we prove by contradiction that $rk(M) \leq 2$ implies condition $(2)$.
Suppose that the equation $M{\bf x} = {\bf 0}$ does not have
a solution without zero coordinates.  By $(2)$, there is at least one
solution to $M{\bf x} = {\bf 0}$,
say ${\bf b}  = \begin{bmatrix} b_0 \\ b_1 \\ b_2 \end{bmatrix}$,
with some $b_i =0$.  Without loss of generality, assume that $b_2 = 0$.
So, for all $i=1,\ldots,r$, we have
\[\frac{b_0}{a_0(i)} + \frac{b_1}{a_1(i)} = \frac{a_1(i)b_0+a_0(i)b_1}{a_0(i)a_1(i)} = 0.\]
We thus have $a_1(i)b_0+a_0(i)b_1 = 0$ for all $i=1,\ldots,r$.
But then every $L_i$ vanishes at
$[b_1:b_0:0]$, contradicting the fact that $\mathbb{X}$ is a star
configuration.   So $M{\bf x} = {\bf 0}$ must have a solution
with no zero-coordinates.
\end{proof}

\begin{remark} Note that to check that condition
\eqref{thmhadamardset2}  holds it is enough to check one of the following statements:
\begin{itemize}
\item the ideal of $\{{\bf x} ~|~ M{\bf x} = {\bf 0}\}
\subset \mathbb{C}^{n+1}$ does
 not contain  the product $\Pi_0^{n}x_i$;
\item there is no linear combination of the rows of $M$ having
exactly one non-zero entry.
\end{itemize}

\end{remark}

\begin{remark}\label{remhadamardset} From  Theorem \ref{thmhadamardset}
it follows that if
$\mathbb{X} = \mathbb{X}_c(L_1,\ldots,L_r)\subset\PP^n$ is a star configuration
with $L_i=a_0(i)x_0+\cdots+a_n(i)x_n$ for all $i$, then $\mathbb{X}$ is a
Hadamard star configuration if and only if
\begin{enumerate}
\item\label{corhadamardset1} $a_j(i)\neq 0$ for all $i,j$; and
\item\label{corhadamardset2} the matrix
\[M=\left(\begin{array}{ccc}{1\over a_0(1)} & \ldots & {1\over a_n(1)} \\ \vdots & \vdots & \vdots\\ {1\over a_0(r)} & \ldots & {1\over a_n(r)}\end{array}\right)\]
is such that there exists an
$\mathbf{x}\in\mathbb{C}^{n+1}$ without zero coordinates
such that $M\mathbf{x}=0$.
\end{enumerate}
 Moreover, if $n=2$,  the condition $(2)$ is equivalent to $\mathrm{rk}(M)\leq 2$.
\end{remark}

%%%%%%%%%%%%%%%%%%%%%%%%%%%%%%%%%%%%%%%%%%%%%%%%%%%%%%%%%%%%%%%%%%%%%%%%%

\section{Toward  explicit constructions of  Hadamard star configurations
in $\PP^n$}\label{sec:building}

In \cite{BCK} the authors show how to construct zero dimensional star
configurations by using a sufficiently general line $\ell\subset\PP^n$.
{In particular,
after one fixes points $P_1,\ldots,P_r\in \ell$ with no zero coordinates,
 one then computes all possible $n$-wise Hadamard products of $n$ of the
points without repetition}. In this way one can explicitly produce the
points of the star configuration { without even knowing
the set $\mathcal{L}$}.   But, if $n\geq 3$, there is no
explicit description of the {\em hyperplanes} defining the star configuration. In what follows, we want to
describe a star configuration in $\mathbb{P}^n$ using a hyperplane $H$ and
points lying on it.

Given a hyperplane $H\subset\PP^n$ and points $P_i\in H$, it is not enough to assume that no zero coordinates appear in order to obtain a star configuration intersecting the hyperplanes $P_i\star H$. For example, it is easy to construct examples in which the hyperplanes $P_1\star H,P_2\star H$ and $P_3\star H$ intersect in codimension smaller than three.
\begin{example}
Consider the plane $H\subset\PP^3$ defined by $4x_0-5x_1+2x_2-x_3=0$ and the points
\[P_1=[1:1:1:1],P_2=[2:1:-2:-1],P_3=[5:4:10:20]\in H.\]
By direct computations we obtain the planes
\[
\begin{aligned}
& H_1=P_1\star H: 4x_0-5x_1+2x_2-x_3=0, \\
& H_2=P_2\star H: 2x_0-5x_1-x_2+x_3=0, \\
& H_3=P_3\star H: 16x_0-25x_1+4x_2-x_3=0,
\end{aligned}
\]
and we see that $H_3\supset H_1\cap H_2$.
\end{example}

\begin{definition}
The {\it coordinate points} $E_0,\ldots ,E_n\in\PP^n$ are the points $E_i$ having exactly one non-zero coordinate in position $i+1$.
\end{definition}

{ The next theorem is the first step towards generalizing the
work of \cite{BCK}.   In particular, the next result shows that
if one mimics the construction of \cite{BCK} for points on a
hyperplane (instead of points on a line), and if one produces
a star configuration,
then there must be a special relationship between the points on
the hyperplane and the coordinate points.

\begin{theorem}\label{thmRNC}
Let $H\subset\PP^n$ be a hyperplane and $P_1,\ldots,P_r\in H$ be points
such that $P_i\not\in\Delta_{n-1}$ and $H$ does not contain any of the
coordinate points $E_i$. Set $H_i=P_i\star H$.

If
\[\mathbb{X}_c(H_1,\ldots, H_r)\]
is a codimension $c\geq 2$ star configuration, then there is no rational normal curve containing the coordinates points $E_0,\ldots ,E_n$ and the points $P_i$, $P_j$, and $P_k$ for all possible choices of $1\leq i<j<k\leq r$.
\end{theorem}
\begin{proof} Assume
$P_i=[p_0(i):\cdots:p_n(i)] \in \mathbb{P}^n \setminus \Delta_{n-1}$, so
$$H_i : \frac{a_0}{p_0(i)}x_0 + \cdots + \frac{a_n}{p_n(i)}x_n =0,$$
where $H = a_0x_0+\cdots+a_nx_n$.

The hyperplane $H_i $ corresponds to the point
$ { H_i }^{\vee}= \left[ \frac{a_0}{p_0(i)}:  \cdots : \frac{a_n}{p_n(i)} \right ] \in ( \mathbb{P}^n)^{\vee}.$ If $\mathbb{X}_c(H_1,\ldots, H_r)$
is a star configuration,  then $H_i$, $H_j$, $H_k$ ($ i<j<k$) meet in
codimension $3$, that is, the points
$ { H_i }^{\vee}$, $ { H_j }^{\vee}$, $ { H_k }^{\vee}$ are not collinear.

Let $\sigma : (\mathbb{P}^n)^{\vee}   \dashedrightarrow (\mathbb{P}^n)^{\vee}$ be
the  standard Cremona trasformation
$ \sigma ([y_0: \cdots :y_n] )
=\left [\frac {1} y_0: \cdots :\frac {1}y_n \right].
$
By the well known properties of Cremona's maps, we have that
$ { H_i }^{\vee}$, $ { H_j }^{\vee}$, $ { H_k }^{\vee}$ are  collinear
if and only if the points $P_i$, $P_j$, $P_k$, $E_0,\ldots ,E_n$ lie on
a rational normal curve, and the result is proved.
\end{proof}

\begin{remark}
Note that Theorem \ref{thmRNC} only gives a necessary condition. In fact,
the non-existence of the rational normal curves assures us that the three
planes $P_i\star H,P_j\star H$, and $P_k\star H$ intersect in codimension
three, that is, their equations are linearly independent. However, we do
not know what happens for four or more planes $H_i$, that is, we do not
know whether their equations are linearly independent. Hence we do not
know whether the collection of {\em all} the linear forms is generally linear.

\end{remark}

\begin{remark}
Note that in the case $n=2$, the necessary  condition of
Theorem \ref {remhadamardset}  is always satisfied since a line
cannot intersect a rational normal curve of  $\PP^2$, that is, an
irreducible conic, in three points. Note that the conic through
$P_i$, $P_j,$ $E_0,$  $E_1,$ $E_2$  is irreducible since the points
$P_i$ and $P_j$ do not have zero coordinates.
\end{remark}

%%%%%%%%%%%%%%%%%%%%%%%%%%%%%%%%%%%%%%%%%%%%%%%%%%%%%%%%%%%%%%%%%%%%

\end{document}